\documentclass[11pt]{amsart}
\usepackage[english]{babel}
\usepackage[latin1]{inputenc}
\usepackage{amssymb,geometry,color,xcolor,graphicx}
\usepackage{amsmath,amsthm}
\usepackage{booktabs}
\usepackage{caption}
\usepackage{datetime}
\usepackage{dsfont}
\usepackage{amstext}
\usepackage{graphicx}
\usepackage{fancyhdr}
\usepackage{latexsym}
\usepackage{mathrsfs}
\usepackage{cases}
\usepackage{mathtools}
\usepackage{bm}
\usepackage{subcaption}
\usepackage{wrapfig}
\usepackage{xcolor,listings}
\usepackage[numbered]{matlab-prettifier}
\usepackage{hyperref}
\hypersetup{
    colorlinks=true, 
    linktoc=all,     
    linkcolor=blue,  
}
\numberwithin{equation}{section}

\newcommand{\udef}{\mathrel{\mathop:}=}

\newcommand{\R}{\mathbb{R}}

\newcommand{\N}{\mathbb{N}}

\usepackage{dsfont}
\newcommand{\1}{\mathds{1}}

\theoremstyle{plain}
\newtheorem{thm}{Theorem}[section]
\newtheorem{cor}[thm]{Corollary}
\newtheorem{prop}[thm]{Proposition}

\newtheorem{defn}[thm]{Definition}
\newtheorem{rem}[thm]{Remark}
\newtheorem{exm}[thm]{Example}

\newcommand{\bmat}[1]{ \begin{bmatrix}#1\end{bmatrix}}

\newcommand{\conv}{\mathrm{conv}}

\renewcommand{\ss}{\scriptstyle}
\def\svdots{\vbox{\baselineskip=1.5pt\lineskiplimit=0pt
	\kern1.5pt \hbox{$\ss .$}\hbox{$\ss .$}\hbox{$\ss .$}}}

\begin{document}

\title{Set-Valued Analysis of Generalized Barycentric Coordinates and Their Geometric Properties}
\author[Difonzo]{Fabio V. Difonzo}
\address{Istituto per le Applicazioni del Calcolo \textquotedblleft Mauro Picone\textquotedblright, Consiglio Nazionale delle Ricerche, Via G. Amendola 122/I, 70126 Bari, Italy}
\email{fabiovito.difonzo@cnr.it}



\null\hfill Version of \today $, \,\,\,$ \xxivtime

\begin{abstract}
Letting $P$ be a convex polytope in $\R^d$ with $n>d$ vertices, we study geometric and analytical properties of the set of generalized barycentric coordinates relative to any point $p\in P$. We prove that such sets are polytopes in $\R^n$ with at most $n-d-1$ vertices, and provide results about continuity and differentiability for the corresponding set-valued maps.
\end{abstract}

\maketitle

\pagestyle{myheadings}
\thispagestyle{plain}
\markboth{F.V. DIFONZO}{SET-VALUED ANALYSIS OF GENERALIZED BARYCENTRIC COORDINATES}

\section{Introduction}

Barycentric coordinates, first introduced by M\"{o}bius \cite{mobius1827barycentrische}, are nowadays useful tools in interpolation problems \cite{floater2003,floater_2015}, computational mechanics \cite{bathe}, computer graphics \cite{SchneiderEtAl2013} and differential equation theory \cite{dd}. We refer to \cite{HormannSukumar2018GBC} and all the references therein for a thorough source of theory and applications of generalized barycentric coordinates. In this work we attempt to study them as set-valued maps from convex $d$-polytopes with $n$ vertices, $n>d$, having values onto the unit cube $[0,1]^n$, and to provide some fundamental geometrical and analytical properties about them.

\section{Setting}

Let $n,d\in\N$, with $n>d$, $\{v_i\}_{i=1}^{n}\subseteq\R^d$ be a set of vertices and let
\[
V\udef\bmat{v_1 & \cdots & v_n}\in\R^{n\times d}.
\]
Defining $\1$ to be the vector whose components are all equal to one, and assuming that $\dim\ker\bmat{V \\ \1^\top}=n-d-1$, then there exists $N\in\R^{n\times(n-d-1)}$ such that
\begin{equation}\label{eq:N}
\langle N\rangle=\ker\bmat{V \\ \1^\top}.
\end{equation}
Let $P\udef\conv V$ be the convex polytope whose vertices are given by the columns of $V$.
\begin{defn}\label{def:gbc}
Let $p\in P$. A vector $\lambda(p)\in\R^n$ is a set of generalized barycentric coordinates for $p$ if it is a componentwise nonnegative solution to the linear system
\begin{equation}\label{eq:gbc}
\bmat{V \\ \1^\top}\lambda=\bmat{p \\ 1}.
\end{equation}
\end{defn}
When $\lambda\geq0$ componentwise, we will also say that $\lambda$ is feasible. It comes from the definition that, for each $p\in P$ there always exists a set of feasible generalized barycentric coordinates $\tau(p)\geq0$. More precisely, since \eqref{eq:gbc} is full rank, it is meaningful to define the set, depending on $p$,
\begin{equation}\label{eq:Gamma}
\Gamma(p)\udef\{c\in\R^{n-d-1}\,:\,\tau(p)+Nc\geq0\},\quad p\in P.
\end{equation}
Thus, we can consider the following set-valued map
\begin{align*}
\Gamma:P &\rightrightarrows \R^{n-d-1} \\
p &\mapsto \Gamma(p)
\end{align*}
It then follows that the entire set of feasible barycentric coordinates relative to $p\in P$ can be expressed by the set-valued map
\begin{equation}\label{eq:Lambda}
\begin{aligned}
\Lambda:P &\rightrightarrows [0,1]^n \\
p &\mapsto \tau(p)+N\Gamma(p).
\end{aligned}
\end{equation}

We prove that $\Lambda(p)$, and hence $\Gamma(p)$, do not depend on the particular $\tau(p)$ chosen. In fact, if $\tau'(p)$ is another feasible set of barycentric coordinates for $p$, then there exists $c'\in\Gamma(p)$ such that
\[
\tau(p)=\tau'(p)+Nc',
\]
and therefore $\Lambda(p)=\tau'(p)+Nc'+N\Gamma(p)=\tau'(p)+N\Gamma(p)$. \\
Let us also notice that for each $c\in\Gamma(p)$ there exists a unique $\lambda\in\Lambda(p)$ such that $\lambda(p)=\tau(p)+Nc$ and, thus, we have a bijective correspondence between the two sets $\Gamma(p)$ and $\Lambda(p)$.

It is easily seen that, from \eqref{eq:gbc},
\[
\Lambda(p)=\left\{\lambda\in\R^n\,:\,\bmat{V \\ \1^\top}\lambda=\bmat{p \\ 1},\,\lambda\geq0\right\},
\]
where, hereafter, $\lambda\geq0$ is meant to hold componentwise. \\
This paper is devoted to study geometric and analytical properties of the set-valued maps $\Gamma(p)$ and $\Lambda(p)$. \\
In Section \ref{sec:geomProp} we prove that $\Lambda(p)$ is a bounded polytope, and we describe its vertices.
In Section \ref{sec:anProp} we study smoothness of these maps.

\section{Geometric structure of $\Lambda(p)$ and $\Gamma(p)$}\label{sec:geomProp}

In this section we state and prove fundamental results about the geometric features of $\Lambda(p)$ and $\Gamma(p)$. Such results will be useful for investigating analytical properties of the set-valued maps introduced above.

It holds the following.
\begin{prop}\label{prop:polytopes}
For each $p\in P$, $\Gamma(p)$ and $\Lambda(p)$ are polytopes of dimension at most $n-d-1$.
\end{prop}
\begin{proof}
From \eqref{eq:Gamma} $\Gamma(p)$ is a closed convex polyhedron and, since $\tau(p)$ can always be chosen such that $\tau(p)\geq0$, $0\in\Gamma(p)$, therefore providing $\Gamma(p)^{\circ\circ}=\Gamma(p)$, where $\Gamma(p)^\circ$ is the polar of $\Gamma(p)$. From classical results (e.g., see \cite{bar,grunbaum2003convex}), $\Gamma(p)^\circ$ is a polyhedron, so that $\Gamma(p)^{\circ\circ}$ is a polytope, and so is $\Gamma(p)$. Since $\Lambda(p)=\tau(p)+N\Gamma(p)$, then $\Lambda(p)$ is a polytope as well. \\
Thus $\Gamma(p)$ is a convex, and hence affine, subspace of $\R^{n-d-1}$, and so $\dim\Gamma(p)\leq n-d-1$. By Caratheodory's Theorem (\cite{bar}) let $\{c_1,\ldots,c_{n-d}\}$ be an affine base for $\Gamma(p)$.

Given an arbitrary $\lambda\in\Lambda(p)$, it can be written as $\lambda=\tau+Nc$ for some unique $c\in\Gamma(p)$. Therefore, there exist $\alpha_i\geq0, i=1,\ldots,n-d$, such that
\[
c=\sum_{i=1}^{n-d}\alpha_ic_i,\quad\sum_{i=1}^{n-d}\alpha_i=1.
\]
Letting $\lambda_i\udef\tau+Nc_i, i=1,\ldots,n-d$, it follows that $\lambda_i\in\Lambda(p)$ and
\[
\lambda=\sum_{i=1}^{n-d}\alpha_i\lambda_i.
\]
Thus $\Lambda(p)\subseteq\conv\{\lambda_i\,:\,i=1,\ldots,n-d\}$, implying $\dim\Lambda(p)\leq n-d-1$.
\end{proof}

\begin{rem}\label{rem:polytopes}
If $p\in\partial P$ and since facet reducing property is shared by all feasible barycentric coordinates, then $\dim\Lambda(p)=0$. \\
If $p\in\mathrm{int}P$ then, by Caratheodory's Theorem, every $\lambda\in\Lambda(p)$ can be written as the convex combination of $n-d$ elements in $\Lambda(p)$.
\end{rem}

\begin{cor}
$\Gamma:P\rightrightarrows\R^{n-d-1}$ and $\Lambda:P\rightrightarrows[0,1]^n$ have compact convex values.
\end{cor}
\begin{proof}
It is a direct consequence of Proposition \ref{prop:polytopes}.
\end{proof}
Next, we want to describe the vertices of the $\Lambda(p)$ as a polytope. We prove that they are given by the simplicial barycentric coordinates that are feasible at $p$, as defined below. In order to prove this, we introduce some notation.

For each $i=1,\ldots,n$, let
\[
\mathcal{I}_i\udef\{i,i+1,\ldots,i+n-d-2\},
\]
where indices are considered circularly (i.e., $n+1\equiv n$ and so on). Let us observe that $|\mathcal{I}_i|=n-d-1$. Let then $\sigma^i(p)$ be the unique solution to
\begin{equation}\label{eq:VISigma}
\bmat{V \\ \1^\top \\ I_i}\sigma=\bmat{p \\ 1 \\ 0},
\end{equation}
where $I_i$ is a diagonal matrix such that
\[
I_i(j,j)\udef
\begin{cases}
1, & j\in\mathcal{I}_i, \\
0, & j\notin\mathcal{I}_i.
\end{cases}
\]
As $\{v_i\,:\,i=1,\ldots,n\}$ are such that $\dim\ker\bmat{V \\ \1^\top}=n-d-1$, it follows that $\bmat{V \\ \1^\top \\ I_i}$ is invertible, and thus \eqref{eq:VISigma} has a unique solution $\sigma^i(p)$, which we refer to as \emph{simplicial barycentric coordinates} for $p$. It is straightforward to realize that
\begin{align*}
\sigma^i_j(p) &=0,\,\,\textrm{ if }j\in\mathcal{I}_i, \\
\sigma^i_j(p) &\neq0,\,\,\textrm{ if }j\notin\mathcal{I}_i.
\end{align*}
It holds the following
\begin{thm}\label{thm:LambdaConvVertices}
Let $p\in P$. Then there exists $I(p)\subseteq\{1,\ldots,n\}$ such that $|I(p)|=n-d$ and $\sigma^i(p)\geq0$ is a vertex of $\Lambda(p)$ for all $i\in I(p)$. Therefore
\[
\Lambda(p)=\conv\{\sigma^i(p)\,:\,i\in I(p)\}.
\]
\end{thm}
\begin{proof}
From Proposition \ref{prop:polytopes}, it suffices to prove that there exist at most, and thus exactly by Caratheodory's Theorem, $n-d$ feasible simplicial barycentric coordinates $\sigma^i(p)$ at $p$. Proceeding by contradiction, let us suppose that for any $\overline n>n-d$ and for any $i=1,\ldots,\overline{n}$, there is some $j_i=1,\ldots,n$ such that $\sigma^i_{j_i}(p)<0$, where $\sigma^i(p)$ is the simplicial barycentric coordinates for $p$ defined in \eqref{eq:VISigma}. This implies that
\[
p\notin S_i\udef\conv\{v_j\,:\,j\notin\mathcal{I}_i\},
\]
for all $i=1,\ldots,\overline n$. Let us note that each $S_i$ is a $d$-simplex. This contradicts the fact that the polytope $P$ has at least $n-d$ simplices (see \cite{Edmonds1970,grunbaum2003convex,Ziegler1995Lectures}). \\
Let then $I(p)\subseteq\{1,\ldots,n\}$ be such that $|I(p)|=n-d$ and $\sigma^i(p)\geq0$ for all $i\in I(p)$, and let us now prove that $\sigma^i(p)$ are vertices of $\Lambda(p)$. In fact, let $i\in I(p)$ be given and let $\lambda,\mu\in\Lambda(p)$ and $a\in(0,1)$ be such that
\[
\sigma^i(p)=(1-a)\lambda+a\mu.
\]
Therefore, for all $j\in\mathcal{I}_i$,
\[
0=(1-a)\lambda_j+a\mu_j,
\]
so that, by nonnegativity, it follows that $\lambda_j=\mu_j=0$, and thus $\lambda,\mu$ both solve \eqref{eq:VISigma}. Since it is nonsingular, we deduce $\lambda=\mu=\sigma^i(p)$, implying that $\sigma^i(p)$ is a vertex of $\Lambda(p)$. Resorting to Krein-Milman Theorem proves the claim.
\end{proof}

\begin{exm}\label{ex:quadR2}
Let us consider a convex quadrilateral $Q\subseteq\R^2$ defined as
\[
Q=\conv\{v_i\,:\,i=1,2,3,4\},
\]
where the $v_i$'s are assumed to be in general position. In this case, letting $p\in Q$, Theorem \ref{thm:LambdaConvVertices} implies that $\Lambda(p)$ is a segment in $\R^4$, whose vertices are given by the triangular barycentric coordinates nonnegative at $p$. In other words, assuming without loss of generality that $p\in\conv\{v_1,v_2,v_3\}\cap\conv\{v_1,v_2,v_4\}$, and letting $\tau^4(p), \tau^3(p)$ be the unique solutions to
\[
\bmat{V \\ \1^\top \\ \begin{matrix} 0 & 0 & 0 & 1 \end{matrix}}\tau^4(p)=\bmat{p \\ 1 \\ 0},
\bmat{V \\ \1^\top \\ \begin{matrix} 0 & 0 & 1 & 0 \end{matrix}}\tau^3(p)=\bmat{p \\ 1 \\ 0}
\]
respectively, then $\Lambda(p)=\conv\{\tau^3(p),\tau^4(p)\}$.

In fact, here $N\in\R^4$ and it is easily seen that for any $\lambda(p)\in\Lambda(p)$ there exists a unique $c\in[a,b]$ such that
\[
\lambda(p)=\tau^4(p)+cN,
\]
where we define $a\udef\max\left\{-\frac{\tau^4_i(p)}{N_i}\,:\,N_i>0\right\}, b\udef\min\left\{-\frac{\tau^4_i(p)}{N_i}\,:\,N_i<0\right\}$ (see \cite{dd2}). It is completely evident that one between $a,b$ must be zero: assuming it to be $a$, then $b>0$; thus $c=a$ provides $\lambda(p)=\tau^4(p)$, while $c=b$ provides $\lambda(p)=\tau^3(p)$. Therefore, since $\tau^4(p)-\tau^3(p)\in\langle N\rangle$, we deduce that $\lambda(p)\in\conv\{\tau^3(p),\tau^4(p)\}$, as expected. If $b=0$, then it is enough, without loss of generality, to consider $-N$ instead of $N$ in the arguments above.
\end{exm}

\begin{exm}\label{ex:pentaR2}
For pentagons or polytopes with more than five vertices in $\R^2$, a similar idea has been explored in \cite{ANISIMOV2017}, where authors propose to consider a suitably blended average of barycentric coordinates relative to a given triangularization of the polytope. As a matter of fact, Theorem \ref{thm:LambdaConvVertices} provides a theoretical justification to the blending approach.

More explicitly, let us consider a convex pentagon $P\subseteq\R^2$ defined as
\[
P\udef\conv\{v_i\,:\,i=1,\ldots,5\},
\]
with $\dim\ker\bmat{V \\ \1^\top}=2$. Let now $\mathcal{T}_{ijk}\udef\conv\{v_l\,:\,l=i,j,k\}$ and $\tau_{ijk}(p)$ be its corresponding triangular barycentric coordinates relative to $p\in\mathcal{T}_{123}\cap\mathcal{P}_{124}\cap\mathcal{P}_{125}$. Thus, $\tau_{123},\tau_{124},\tau_{125}$ are, respectively, the unique solutions to
\[
\bmat{V \\ \1^\top \\ \begin{matrix} 0 & 0 & 0 & 1 & 0 \end{matrix} \\ \begin{matrix} 0 & 0 & 0 & 0 & 1 \end{matrix}}\tau=\bmat{p \\ 1 \\ 0 \\ 0},
\bmat{V \\ \1^\top \\ \begin{matrix} 0 & 0 & 1 & 0 & 0 \end{matrix} \\ \begin{matrix} 0 & 0 & 0 & 0 & 1 \end{matrix}}\tau=\bmat{p \\ 1 \\ 0 \\ 0},
\bmat{V \\ \1^\top \\ \begin{matrix} 0 & 0 & 1 & 0 & 0 \end{matrix} \\ \begin{matrix} 0 & 0 & 0 & 1 & 0 \end{matrix}}\tau=\bmat{p \\ 1 \\ 0 \\ 0}.
\]
It is easy to see that $\tau_{123},\tau_{124},\tau_{125}$ are linearly independent and $\tau_{124}(p)-\tau_{123}(p),\tau_{125}-\tau_{123}$ span $\ker\bmat{V \\ \1^\top}$. Thus, every feasible generalized barycentric coordinates at $p$, say $\lambda(p)$, can be written as
\[
\lambda(p)=\tau_{123}(p)+c(\tau_{124}(p)-\tau_{123}(p))+\overline{c}(\tau_{125}(p)-\tau_{123}(p)),
\]
for some $c,\overline{c}\in\R$, which implies
\[
\lambda(p)=(1-c-\overline{c})\tau_{123}(p)+c\tau_{124}(p)+\overline{c}\tau_{125}(p).
\]
However, since $\lambda(p)\geq0$ component-wise and because of the structure of $\tau_{ijk}(p)$, we get that $c,\overline{c}\in[0,1]$, so that $\lambda(p)$ is a convex combination of $\tau_{123},\tau_{234},\tau_{345}$. Therefore $\Lambda(p)=\conv\{\tau_{123},\tau_{234},\tau_{345}\}$, which configures a triangle in $\R^5$, thus confirming Theorem \ref{thm:LambdaConvVertices}.
\end{exm}

\begin{exm}\label{ex:pyrR3}
Let us consider a pyramid $P$ in $\R^3$ defined as
\[
P\udef\conv\{v_i\,:\,i=1,\ldots,5\},
\]
with $\dim\ker\bmat{V \\ \1^\top}=1$. This case is analogous to the quadrilateral in Example \ref{ex:quadR2}, and again Theorem \ref{thm:LambdaConvVertices} holds straightforwardly. \\
Let us now consider a prism $P$ in $\R^3$  defined as
\[
P\udef\conv\{v_i\,:\,i=1,\ldots,8\},
\]
with $\dim\ker\bmat{V \\ \1^\top}=4$. In this case, it is immediate to observe that the barycentric coordinates $\sigma_{ijkl}$ relative to the simplices $\conv\{v_m\,:\,m=i,j,k,l\}$, with $ijkl\in I\udef\{1235,1236,1237,1238,1245\}$, span $\ker\bmat{V \\ \1^\top}$ and thus, again resorting to arguments similar as in Example \ref{ex:pentaR2} by counting and exploiting the zero structure of such barycentric coordinates, we deduce that $\Lambda(p)=\conv\{\sigma_{ijkl}\,:\,ijkl\in I\}$.
\end{exm}

Let us notice that the vertices of $\Lambda(p)$ are strictly related to $p\in P$, and thus the set $I(p)$ is expected to vary with $p$. In fact, it is a simple observation that if $p\in\partial S_i$ for some $i=1,\ldots,n$, then there must exist some $j\notin I(p)$ such that $p\in\partial S_j$: for such a point, then, $\sigma^i(p)=\sigma^j(p)$. \\ 
Such a continuity of $\Lambda(p)$ is related to the selection problem, i.e. all the possible sets of generalized barycentric coordinates. Because of \eqref{eq:Lambda}, continuity of $\Gamma(p)$ and $\Lambda(p)$ imply each other, provided that $\tau(p)$ in \eqref{eq:Lambda} is smooth enough. Since $\Lambda(p)$ does not depend on the particular $\tau(p)$ chosen and simplicial barycentric coordinates are always smooth and feasible as long as $p$ belongs to their corresponding simplices, we will always assume $\tau(p)$ to be as smooth as needed. We are going to further deepen this topic in the next section.

\section{Analytical Properties of $\Lambda(p)$ and $\Gamma(p)$}\label{sec:anProp}

In this section we investigate smoothness of the generalized barycentric coordinates maps.

The following results describe some analytical properties of $\Gamma(p)$ and $\Lambda(p)$.

\begin{prop}\label{prop:LambdaUSC}
$\Lambda:P\rightrightarrows[0,1]^n$ is upper semi-continuous.
\end{prop}
\begin{proof}
Since $P$ is closed and $[0,1]^n$ is compact, we only need to prove that $\mathrm{Graph}(\Lambda)$ is closed (see \cite{aubinFrankowska2009set}). Letting $(p,\lambda)\in\overline{\mathrm{Graph}(\Lambda)}$, it follows that there exist $\{p_k\}\subseteq P,\{\lambda_k\}\subseteq[0,1]$ such that
\[
\lim_{k\to\infty}p_k=p,\,\,\lim_{k\to\infty}\lambda_k=\lambda,
\]
and
\[
\lambda_k\in\Lambda(p_k),\,\,k\in\N.
\]
Therefore $V\lambda_k=p_k$ for all $k\in\N$ and so, taking the limit of both sides, $V\lambda=p$, implying $\lambda\in\Lambda(p)$. It then holds that $\mathrm{Graph}(\Lambda)$ is closed and thus $\Lambda$ is upper semi-continuous.
\end{proof}

\begin{prop}\label{prop:LambdaLSC}
$\Lambda:P\rightrightarrows[0,1]^n$ is lower semi-continuous.
\end{prop}
\begin{proof}
Let $p\in P$ and $\{p_k\}\in P$ such that $p_k\to p$ as $k\to\infty$; let also $\lambda\in\Lambda(p)$. Therefore, for each $k\in\N$ there exists $\lambda_k\in[0,1]^n$ such that $p_k=V\lambda_k$ and $\1^\top\lambda_k=1$. Thus
\[
\lim_{k\to\infty}V\lambda_k=V\lambda,
\]
and so, since $V$ is full column rank, it follows from standard arguments that $\{\lambda_k\}$ converges and
\[
\lim_{k\to\infty}\lambda_k=\lambda.
\]
Since $\lambda_k\in\Lambda(p_k)$ by construction, the claim is proved.
\end{proof}
It is a direct consequence of definition that $\Gamma(p)$ is also upper and lower semi-continuous. We then have the following.
\begin{cor}
$\Lambda:P\rightrightarrows[0,1]^n$ and $\Gamma:P\rightrightarrows\R^{n-d-1}$ are continuous set-valued maps.
\end{cor}
\begin{proof}
It is a straightforward consequence of Propositions \ref{prop:LambdaUSC}, \ref{prop:LambdaLSC} and definition of $\Gamma$.
\end{proof}

We want to focus now on smoothness features of $\Lambda(p)$. In particular, we are going to prove that it is semidifferentiable in $P$. We recall here some useful definitions for the next results (see \cite{Penot1984,DENTCHEVA1998} for further details). \\
Let $X$ be a normed vector space and $F:X\rightrightarrows\R^n$ be a set-valued map with nonempty images.
\begin{defn}
The set-valued map $F:X\rightrightarrows\R^n$ is called \emph{semidifferentiable} at a point $(\overline x,\overline y)\in\mathrm{Graph}F$ if the limit
\[
DF(\overline x,\overline y;h_0)\udef\lim_{t_n\to0^+,\,h_n\to h_0}t_n^{-1}\left[F(\overline x+t_nh_n)-\overline y\right]
\]
exists for all $h_0\in X$ in the sense of Kuratowski-Painlev\'{e} (see \cite{aubinFrankowska2009set}).
\end{defn}
\begin{thm}\label{thm:semidiff}
$\Lambda:P\rightrightarrows[0,1]^n$ is semidifferentiable at every point $(p,\lambda(p))\in\mathrm{Graph}F$ where $\lambda(p)$ is differentiable.
\end{thm}
\begin{proof}
Let $p\in P$ and $h\in P$ be given. It is known \cite{Penot1984} that in order to prove that $\Lambda$ is semidifferentiable we need to prove that $D\Lambda(p,\lambda(p);h)$ is nonempty, which requires that there exists $v\in\R^n$ such that
\begin{equation}\label{eq:liminf}
\liminf_{(t_n,h_n)\to(0^+,h)}t_n^{-1}d(\lambda(p)+t_nv,\Lambda(p+t_nh_n))=0,
\end{equation}
where $d(a_0,A)\udef\inf_{a\in A}\|a_0-a\|$. By definition, \eqref{eq:liminf} is equivalent to
\[
\liminf_{(t_n,h_n)\to(0^+,h)}t_n^{-1}d(\lambda(p)+t_nv,\lambda(p+t_nh_n)+N\Gamma(p+t_nh_n))=0,
\]
and thus
\[
\liminf_{(t_n,h_n)\to(0^+,h)}t_n^{-1}d(\lambda(p)-\lambda(p+t_nh_n)+t_nv,N\Gamma(p+t_nh_n))=0.
\]
Using homogeneity and continuity of the distance function and since $\lambda(p)$ is differentiable, we have to prove that
\begin{equation}\label{eq:liminf_d}
\liminf_{(t_n,h_n)\to(0^+,h)}d(-J\lambda(p)h_n+v,t_n^{-1}N\Gamma(p+t_nh_n))=0.
\end{equation}
Since, from Proposition \ref{prop:polytopes}, $\Gamma(p+t_nh_n)$ is a bounded polytope containing the origin for all $n\in\N$, then $0_{\R^n}\in t_n^{-1}N\Gamma(p+t_nh_n)$ for all $n\in\N$. Thus, for any given $\varepsilon>0$ sufficiently small we have
\[
B_\varepsilon(0)\subseteq t_n^{-1}N\Gamma(p+t_nh_n).
\]
Let now $J\lambda(p)$ be the Jacobian of $\lambda$ at $p$. Since
\[
\lim_{n\to\infty}J\lambda(p)h_n=J\lambda(p)h,
\]
then there exists $\overline{n}\in\N$ such that for all $n\geq\overline{n}$ it holds $\|J\lambda(p)h_n-J\lambda(p)h\|\leq\varepsilon$. This implies that
\[
J\lambda(p)h-J\lambda(p)h_n\in B_\varepsilon(0)\subseteq t_n^{-1}N\Gamma(p+t_nh_n)
\]
for all $n\geq\overline{n}$, providing that \eqref{eq:liminf_d} is satisfied for $v\udef J\lambda(p)h$, and this proves the claim.
\end{proof}


\section{Conclusions}

In this paper we have considered the set-valued map of generalized barycentric coordinates $\Lambda(p)$ relative to a $d$-polytope $P$ with $n>d$ vertices, and we have proved that, for any given point $p\in P$, values of $\Lambda(p)$ define a polytope of dimension at most $n-d-1$, whose vertices are represented by $n-d$ nonnegative simplicial barycentric coordinates relative to $p$. This result seems to be a dual transform of a convex polytope $P\subseteq\R^d$, in the sense that as every point $p\in P$ can be described by a convex combination of the $n$ vertices of $P$, so each of the barycentric coordinates relative to $p$ can be described by a convex combination of $n-d$ simplicial barycentric coordinates in $\R^n$. \\
Moreover, we have proved that the these maps are continuous and semidifferentiable at points where the corresponding generalized barycentric coordinate is differentiable.

\section*{Acknowledgments}
FVD is part of INdAM research group GNCS. He also gratefully thanks Prof. L.~Dieci and Prof. N.~Sukumar for their support while writing this paper.

\bibliographystyle{plain}
\bibliography{BibliografiaFabio.bib}

\begin{thebibliography}{10}

\bibitem{ANISIMOV2017}
Dmitry Anisimov, Daniele Panozzo, and Kai Hormann.
\newblock Blended barycentric coordinates.
\newblock {\em Computer Aided Geometric Design}, 52-53:205--216, 2017.
\newblock Geometric Modeling and Processing 2017.

\bibitem{aubinFrankowska2009set}
J.P. Aubin and H.~Frankowska.
\newblock {\em Set-Valued Analysis}.
\newblock Modern Birkh{\"a}user Classics. Birkh{\"a}user Boston, 2009.

\bibitem{bar}
A.~Barvinok.
\newblock {\em A {C}ourse in {C}onvexity}.
\newblock American {M}athematical {S}ociety, 2007.

\bibitem{bathe}
K.J. Bathe.
\newblock {\em Finite Element Procedures}.
\newblock Prentice Hall, 1996.

\bibitem{DENTCHEVA1998}
Darinka Dentcheva.
\newblock {Differentiable Selections and Castaing Representations of
  Multifunctions}.
\newblock {\em Journal of Mathematical Analysis and Applications},
  223(2):371--396, 1998.

\bibitem{dd}
L.~Dieci and F.~Difonzo.
\newblock A comparison of {F}ilippov sliding vector fields in codimension 2.
\newblock {\em Journal of Computational and Applied Mathematics}, 262:161 --
  179, 2014.

\bibitem{dd2}
L.~Dieci and F.~Difonzo.
\newblock The {M}oments {S}liding {V}ector {F}ield on the {I}ntersection of
  {T}wo {M}anifolds.
\newblock {\em Journal of Dynamics and Differential Equations}, 2015.

\bibitem{Edmonds1970}
Allan~L. Edmonds.
\newblock Simplicial decompositions of convex polytopes.
\newblock {\em Pi Mu Epsilon Journal}, 5(3):124--128, 1970.

\bibitem{floater_2015}
M.~S. Floater.
\newblock Generalized barycentric coordinates and applications.
\newblock {\em Acta Numerica}, 24:161--214, 2015.

\bibitem{floater2003}
M.S. Floater.
\newblock Mean {V}alue {C}oordinates.
\newblock {\em Computer Aided Geometric Design}, (20):19--27, 2003.

\bibitem{grunbaum2003convex}
Branko Gr\"{u}nbaum, Volker Kaibel, Victor Klee, and G\"{u}nter~M. Ziegler.
\newblock {\em Convex polytopes}.
\newblock Springer, New York, 2003.

\bibitem{HormannSukumar2018GBC}
Kai Hormann and N~Sukumar.
\newblock {\em Generalized barycentric coordinates in computer graphics and
  computational mechanics}.
\newblock CRC press, Boca Raton, FL, 2017.

\bibitem{mobius1827barycentrische}
A.F. M{\"o}bius.
\newblock {\em {Der barycentrische Calcul}}.
\newblock J.A. Barth, 1827.

\bibitem{Penot1984}
Jean-Paul Penot.
\newblock {Differentiability of Relations and Differential Stability of
  Perturbed Optimization Problems}.
\newblock {\em SIAM Journal on Control and Optimization}, 22(4):529--551, 1984.

\bibitem{SchneiderEtAl2013}
T.~Schneider, K.~Hormann, and M.~S. Floater.
\newblock {Bijective Composite Mean Value Mappings}.
\newblock In {\em Proceedings of the Eleventh Eurographics/ACMSIGGRAPH
  Symposium on Geometry Processing}, SGP '13, pages 137--146. Eurographics
  Association, 2013.

\bibitem{Ziegler1995Lectures}
G\"{u}nter~M. Ziegler.
\newblock {\em Lectures on polytopes}.
\newblock Springer-Verlag, New York, 1995.

\end{thebibliography}

\end{document}